\documentclass[a4paper,10pt]{article}
\usepackage{amsfonts,amsmath}
\usepackage[utf8]{inputenc}
\usepackage[T1]{fontenc}
\usepackage{graphicx}
\usepackage{geometry}
\usepackage{float}
\usepackage{amsmath}
\usepackage{amsthm}
\usepackage{amsfonts}
\usepackage{amssymb}
\usepackage{MnSymbol}
\usepackage{mathdots}
\usepackage{stmaryrd}
\usepackage{enumitem}
\usepackage{bm}
\usepackage[]{caption}
\newtheorem{prop}{Proposition}[section]
\newtheorem{lemma}{Lemma}[section]
\newtheorem{thm}{Theorem}[section]
\newtheorem{cor}{Corollary}[section]
\newtheorem{rem}{Remark}[section]

\newcommand{\w}{\varpi}
\renewcommand{\l}{\lambda}
\newcommand{\tr}{\mathrm{Tr}}
\renewcommand{\det}{\mathrm{det}}
\newcommand{\diag}{\mathrm{diag}}
\newcommand{\e}{\epsilon}
\newcommand{\ve}{\varepsilon}
\newcommand{\C}{\mathbb{C}}
\renewcommand{\d}{\delta}

\newcommand{\BC}{\mathrm{BC}}
\newcommand{\AI}{\mathrm{AI}}

\newcommand{\Hom}{\mathrm{Hom}}

\newcommand{\Res}{\mathrm{Res}}
\newcommand{\Rep}{\mathrm{Rep}}
\renewcommand{\tr}{\mathrm{tr}}
\newcommand{\GL}{\mathrm{GL}}
\newcommand{\LLC}{\mathrm{LLC}}
\newcommand{\SL}{\mathrm{SL}}

\newcommand{\Ind}{\mathrm{Ind}}
\newcommand{\M}{\mathcal{M}}

\newcommand{\1}{\mathbf{1}}

\title{Local distinction, quadratic base change and automorphic induction for $\GL_n$}
\author{N. Matringe}

\begin{document}

\maketitle

\begin{abstract}
Behind this sophisticated title hides an elementary exercise on Clifford theory for index two subgroups and self-dual/conjugate-dual representations. 
When applied to semi-simple representations of the Weil-Deligne group $W'_F$ of a non Archimedean local field $F$, and further translated in terms of representations of $\GL_n(F)$ via the local Langlands correspondence when $F$ has characteristic zero, it yields various statements concerning the behaviour of different types of distinction under quadratic base change and automorphic induction. When $F$ has residual characteristic different from $2$, combining of one of the simple results that we obtain with the tiviality of conjugate-orthogonal root numbers (\cite{GGP}), we recover without using the LLC a result of Serre on the parity of the Artin conductor of orthogonal representations of $W'_F$ (\cite{Serre}). On the other hand we discuss its parity for symplectic representations using the LLC and the Prasad and Takloo-Bighash conjecture. 
\end{abstract}

\section*{Introduction}

Let $E/F$ be a separable quadratic extension of non Archimedean local fields. Then thanks to the known local Langlands correspondence for $\GL_n(E)$ and $\GL_n(F)$, one has a base change map $\BC_F^E$ from the set of isomorphism classes of irreducible representations of $\GL_n(F)$ to 
that of $\GL_n(E)$, and an automorphic induction map $\AI_E^F$ from the set or isomorphism classes of irreducible representations of $\GL_n(E)$ to 
that of $\GL_{2n}(F)$. A typical statement proved in this note (for $F$ of characteristic zero) is that if $\pi$ is a generic unitary representation of $\GL_n(F)$ with orthogonal Langlands parameter (orthogonal in short), then 
$\BC_F^E(\pi)$ is orthogonal and $\GL_n(F)$-distinguished, and that the converse holds if $\pi$ is a discrete series (see Corollary 
\ref{cor main} for the general statement). Corollary \ref{cor main} is itself a translation via the LLC of our main result which concerns 
representations of the Weil-Deligne group of $F$ (Proposition \ref{prop main}). Another lucky application of Proposition \ref{prop main} is that the result of \cite{Serre} on the parity of Artin conductors of representations of the Weil-Deligne group of $F$ is a consequence of that in \cite{Deligne} on root numbers of orthogonal representations, when $F$ has odd residual characteristic, as we show in Corollary \ref{cor serre}. 
We also discuss its parity for symplectic representations using the LLC and the Prasad and Takloo-Bighash conjecture in Corollary \ref{cor PTB}.\\

\noindent \textbf{Acknowledgement.} The motivation for writing this note is a question of Vincent Sécherre, which it answers. We thank him for asking it. We also thank Eyal Kaplan for useful explanations concerning \cite{Kap} and \cite{Yam}. We are greatful to the referee for his accurate comments and corrections. This work benefited from hospitality of the Erwin-Schrödinger institute, via the Research in Teams Project: l-modular Langlands Quotient Theorem and Applications.

\section{Notation, definitions and basic facts about self-dual and conjugate-dual representations}

For $K$ a non Archimedean local field we denote by $W_K$ the Weil group of $K$ (see \cite{T.79}), and by $W'_K=W_K\times \SL_2(\C)$ the Weil-Deligne group of $K$. By a representation of $W_K$ we mean a finite dimensional smooth complex representation of $W_K$. By a representation of $W'_K$ we mean a representation which is a direct sum of representations of the form $\phi\otimes S$, where $\phi$ is an irreducible representation of $W_K$ and $S$ is an irreducible algebraic representation of $\SL_2(\C)$. We sometimes abbreviate "$\phi$ is a representation of 
$W'_K$" as "$\phi\in \Rep(W'_K)$". We denote by $\phi^\vee\in \Rep(W'_K)$ the dual of $\phi\in \Rep(W'_K)$.\\

For the following facts on self-dual and conjugate-dual representations of $W'_K$, we refer to \cite[Section 3]{GGP}. We recall that a representation $\phi$ of $W'_K$ is self-dual if and only if there exists on $\phi\times \phi$ a $W'_K$-invariant bilinear form $B$ which is non degenerate: we will say that $B$ is \textit{$W'_K$-bilinear} (which in particular means non degenerate). If moreover $B$ is alternate, we say that $B$ is \textit{$(W'_K,-1)$-bilinear} in which case we say that $\phi$ is \textit{symplectic} or \textit{$(-1)$-self-dual}, whereas if $B$ is symmetric, and we say that $B$ is \textit{$(W'_K,1)$-bilinear} in which case we say that $\phi$ is \textit{orthogonal} or \textit{$1$-self-dual}. 
If $\phi$ is irreducible and self-dual, then there is up to nonzero scaling a unique $W'_K$-bilinear form on $\phi\times \phi$, which is either 
$(W'_K,-1)$-bilinear or $(W'_K,1)$-bilinear, but not both.\\

Now suppose that $L/K$ is a separable quadratic extension so that $W_L$ has index two in $W_K$, and fix $s\in W_K-W_L$. For $\phi$ a representation of $W'_L$, we denote by $\phi^s$ the representation of $W'_L$ defined as $\phi^s:=\phi(s \ . \ s^{-1})$. We say that $\phi$ is 
\textit{$L/K$-dual} or \textit{conjugate-dual} if $\phi^s\simeq \phi^\vee$. The representation $\phi\in \Rep(W'_L)$ is conjugate-dual if and only if there is on $\phi\times \phi$ a non-degenerate bilinear form $B$ such that \[B(w.x, sws^{-1}.y)=B(x, y)\] for all 
$(w,x,y)$ in $W'_L\times  \phi\times \phi$. We say that such a bilinear form $B$ is \textit{$L/K$-bilinear} (this in particular means non degenerate). If moreover there is $\ve\in \{\pm 1\}$ such that $B$ satisfies 
\[B(x,s^2.y)=\ve B(y,x)\] for all $(x,y)$ in $\phi\times \phi$ we say that $B$ is \textit{$(L/K,\ve)$-bilinear}, in which case we say that 
$\phi$ is \textit{$(L/K,\ve)$-dual} or \textit{conjugate-symplectic} if $\ve=-1$ and \textit{conjugate-orthogonal} if $\ve=1$. All the definitions above do not depend on the choice of $s$. When $\phi$ is $L/K$-dual and also irreducible, then there is up to nonzero scaling a unique $L/K$-bilinear form on $\phi\times \phi$, which is either 
$(L/K,-1)$-bilinear or $(L/K,1)$-bilinear, but not both.

\section{Preliminary results}

\subsection{Clifford-Mackey theory for index two subgroups}

We refer to \cite[Section 3]{CST} for the following standard results.

\begin{thm}\label{clifford}
Let $G$ be a finite group, $H$ a finite subgroup of index $2$, $s\in G-H$, and let $\eta:G\rightarrow \{\pm 1\}$ be the nontrivial character of $G$ trivial on $H$.
\begin{itemize}
\item For $\phi$ a (finite dimensional complex) representation of $H$ which is irreducible, the representation $\Ind_H^G(\phi)$ is irreducible if and only if $\phi^s\not \simeq \phi$, which is also equivalent to the fact that $\phi$ does not extend to $G$. If it is reducible then $\phi$ extends to $G$, and if 
$\widetilde{\phi}$ is such an extension, then $\eta\otimes \widetilde{\phi}$ is the only other extension different from $\widetilde{\phi}$, and $\Ind_H^G(\phi)\simeq \widetilde{\phi}\oplus (\eta\otimes \widetilde{\phi})$.
\item An irreducible representation $\phi'$ of $G$ restricts to $H$ either irreducibly, or breaks into two irreducible pieces, and the second case occurs if and only if $\phi' \simeq \eta \otimes \phi'$, which is also equivalent to $\phi'=\Ind_H^G(\phi)$ for $\phi$ an irreducible representation of 
$H$ such that $\phi^s \simeq \phi$.
\end{itemize}
\end{thm}

For $E/F$ a separable quadratic extension of non Archimedean local fields, we denote by 
$\eta_{E/F}:W'_F\rightarrow \{\pm 1\}$ the nontrivial character of $W'_F$ trivial on $W'_E$. Theorem \ref{clifford} has the following corollary.

\begin{cor}\label{clifford2}
Let $E/F$ be a separable quadratic extension of non Archimedean local fields, and fix $s\in W_F-W_E$. 
\begin{itemize}
\item For $\phi_E\in \Rep(W'_E)$ an irreducible representation, the representation $\Ind_{W'_E}^{W'_F}(\phi_E)$ is irreducible if and only if $\phi_E^s\not \simeq \phi_E$, which is also equivalent to the fact that $\phi_E$ does not extend to $W'_F$. If it is reducible then $\phi_E$ extends to $W'_F$, and if 
$\phi_F$ is such an extension, then $\eta_{E/F} \otimes \phi_F$ is the only other extension different from $\phi_F$, and $\Ind_{W'_E}^{W'_F}(\phi_E)\simeq\phi_F\oplus (\eta_{E/F}\otimes \phi_F)$.
\item An irreducible representation $\phi_F$ of $W'_F$ restricts to $W'_E$ either irreducibly, or breaks into two irreducible pieces, and the second case occurs if and only if $\phi_F \simeq \eta_{E/F} \otimes \phi_F$, which is also equivalent to $\phi_F \simeq \Ind_{W'_E}^{W'_F}(\phi_E)$ for $\phi_E$ and irreducible representation of 
$W'_E$ such that $\phi_E^s \simeq \phi_E$.
\end{itemize}
\end{cor}
\begin{proof}
We recall that by \cite[28.6]{BH.06}, if $\alpha_{K}$ is an irreducible representation of $W_K$ for $K$ local and non Archimedean, then there exists an unramified character $\chi_K$ of $W_K$ such that $\chi_K\otimes \alpha_K$ has co-finite kernel.
 
For the first part of the first point, write $\phi_E=\alpha_E\otimes S$, and suppose first that $\Ind_{W'_E}^{W'_F}(\phi_E)$ is irreducible. Twist 
$\Ind_{W_E}^{W_F}(\alpha_E)$ by an unramified character $\chi_F$ so that $\Ind_{W_E}^{W_F}(\Res_{W_F}^{W_E}(\chi_F)\otimes \alpha_E)$ has a co-finite kernel (hence $\Res_{W_E}^{W_F}(\chi_F) \otimes \alpha_E$ has co-finite kernel as well, as it has to be trivial on $W_E\cap \mathrm{Ker}(\Ind_{W_E}^{W_F}(\Res_{W_E}^{W_F}(\chi_E)\otimes \alpha_E))$). Because $\Res_{W_E}^{W_F}(\chi_F)^s=\Res_{W_E}^{W_F}(\chi_F)$, 
one deduces from Theorem \ref{clifford} applied to $\Res_{W_E}^{W_F}(\chi_E)\otimes \alpha_E$ that $\alpha_E^s\not \simeq \alpha_E$ and that $\alpha_E$ does not extend. This implies the same statements for $\phi_E$. Conversely if $\phi_E^s\not \simeq \phi_E$, then the same holds for $\alpha_E$. Take $\chi_E$ unramified such that $\chi_E\otimes \alpha_E$ has co-finite kernel, and $\chi_F$ any unramified extension of $\chi_E$ to $W_F$. Then $\Ind_{W_E}^{W_F}(\alpha_E)=\chi_F^{-1}\otimes \Ind_{W_E}^{W_F}(\chi_E \otimes \alpha_E)$ is irreducible by Theorem \ref{clifford}, and so is $\Ind_{W'_E}^{W'_F}(\phi_E)=\Ind_{W_E}^{W_F}(\alpha_E)\otimes S$. The second part of the first point is similar, using an unramified character $\chi_F$ of $W_F$ such that $\chi_E\otimes \phi_E$ has cofinite kernel (just take such a $\chi_E$ and extend it to an unramified character of $W_F$). 

The proof of the second point is similar. 

\end{proof}

We will tacitly use the above corollary from now on.

\subsection{Distinction and LLC for $\GL_n$}\label{section LLC}

Let $F$ be a non Archimedean local field, we denote by LLC the local Langlands correspondence (\cite{LRS}, \cite{HT}, \cite{Hen}). For any 
$n\geq 1$, it restricts as a bijection from the set of isomorphism classes 
of $n$-dimensional representations of $W'_F$ to that of (smooth and complex) irreducible representations of $\GL_n(F)$. If $E/F$ is a quadratic extension, and $\pi=\LLC(\phi_F)$ for $\phi_F$ a representation of $W'_F$, we set $\BC_F^E(\pi)= \LLC(\Res_{W'_E}^{W'_F}(\phi_F))$ (the \textit{quadratic base change} of $\pi$), whereas if 
$\tau=\LLC(\phi_E)$ for $\phi_E$ a representation of $W'_E$, we set $\AI_E^F(\tau)= \LLC(\Ind_{W'_E}^{W'_F}(\phi_E))$ (the \textit{quadratic automorphic induction} of $\tau$).
 For $\pi$ a representation of $\GL_n(F)$, we denote by $\pi^\vee$ its dual. If $\pi$ is irreducible, we call it a \textit{discrete series} representation if it has a matrix coefficient $c$ such that 
$|\chi\otimes c|^2$ is integrable on $\GL_n(F)/F^\times .I_n$ (with respect to any Haar measure on the group $\GL_n(F)/F^\times .I_n$) for some character $\chi$ of $\GL_n(F)$. A representation $\phi$ of $W'_F$ is irreducible if and only if $\LLC(\phi)$ is 
a discrete series.\\

Let $N_n(F)$ be the subgroup of $\GL_n(F)$ of upper triangular unipotent matrices, and let $\psi$ be a non trivial character of $F$, which in turn defines a character $\widetilde{\psi}:u\mapsto \psi(u_{1,2}+\dots+u_{n-1,n})$ of $N_n(F)$. We say that an irreducible representation 
$\pi$ of $\GL_n(F)$ is generic if $\Hom_{N_n(F)}(\pi,\widetilde{\psi})\neq \{0\}$ and this does not depend on the choice of $\psi$. Genericity can be read on the Langlands parameter from \cite[Theorem 9.7]{Z} (one way to state it is that $\LLC(\phi)$ is generic if and only if the adjoint L factor of $\phi$ is holomorphic at $s=1$). From this one easily deduces the direct implications of the following proposition, the converse implications being special cases of \cite[Theorem 9.1]{MS}.

\begin{prop}\label{prop MS}
\begin{itemize}
\item Let $\pi$ be an irreducible representation of $\GL_n(F)$. If $\BC_F^E(\pi)$ is generic, then $\pi$ is generic, and conversely if 
$\pi$ is generic unitary, then $\BC_F^E(\pi)$ is generic (unitary).
\item Let $\tau$ be an irreducible representation of $\GL_n(E)$. If $\AI_E^F(\tau)$ is generic, then $\tau$ is generic, and conversely if 
$\tau$ is generic unitary, then $\AI_E^F(\tau)$ is generic (unitary).
\end{itemize}
\end{prop}

We denote by $\widetilde{\GL_n}(F)$ the double cover of $\GL_n(F)$ defined for example in \cite[Section 2.1]{Kap}. Following \cite{Kap} we call 
a map $\gamma:F^\times \rightarrow \C^\times$ a pseudo-character if it satisfies $\gamma(xy)=\gamma(x)\gamma(y)(x,y)_2^{\lfloor n/2 \rfloor}$ for all $x$ and $y$ in $F^\times$, where $(\ ,\ )_2$ is the Hilbert symbol of $F^\times$. For $\gamma$ a pseudo-character of $F^\times$ we denote by $\theta_{1,\gamma}$ the corresponding Kazhdan-Patterson exceptional representation of $\widetilde{\GL_n}(F)$ as in \cite[Section 2.5]{Kap}.
We say that an irreducible representation $\pi$ of $\GL_n(F)$ is \textit{$\Theta_F$-distinguished} if there exist pseudo-characters $\gamma$ and $\gamma'$ of $F^\times$ such that 
$\Hom_{\GL_n(F)}(\theta_{1,\gamma}\otimes \theta_{1,\gamma'},\pi^\vee)\neq\{0\}$ (where $\theta_{1,\gamma}\otimes \theta_{1,\gamma'}$ indeed factors through $\GL_n(F)$ so that the defintion makes sense).\\

When $n$ is even, we denote by $S_n(F)$ the Shalika subgroup of $\GL_n(F)$ consisting of matrices of the form 
$s(g,x)=\diag(g,g)\begin{pmatrix} I_{n/2} & x \\ & I_{n/2} \end{pmatrix}$ for $g\in \GL_{n/2}(F)$ and $x\in \M_{n/2}(F)$, and for $\psi$ a non trivial character of $F$, we denote by $\Psi$ the character of $S_n(F)$ defined by $\Psi(s(g,x))=\psi(\tr(x))$.
We say that an irreducible representation $\pi$ of $\GL_n(F)$ is \textit{$\Psi_F$-distinguished} if $n$ is even and 
$\Hom_{S_n(F)}(\pi,\Psi)\neq\{0\}$ for some non trivial character $\psi$ of $F$. This does not depend on the choice of $\psi$.\\

Finally if $E/F$ is quadratic separable, identifying $\eta_{E/F}$ to the character of $F^\times$ trivial on $N_{E/F}(E^\times)$ via local class field theory, 
we say that an irreducible representation $\tau$ of $\GL_n(E)$ is \textit{$\1_{E/F}$-distinguished} if 
$\Hom_{\GL_n(F)}(\tau,\1)\neq\{0\}$ and \textit{$\eta_{E/F}$-distinguished} if 
$\Hom_{\GL_n(F)}(\tau,\eta_{E/F}\circ \det)\neq\{0\}$.\\

The following theorem follows from \cite{HenExt}, \cite{Kab}, \cite{AKT}, \cite{AR}, \cite{M.11}, \cite{KR}, \cite{Jo}, \cite{MatShal}, \cite{Yam}, \cite{Kap}. Parts of it are known 
to hold when $F$ is of positive characteristic and odd residual characteristic (\cite[Appendix A]{AKMSS}).

\begin{thm}
Suppose that $F$ has characteristic zero. 
\begin{itemize}
\item Let $\pi=\LLC(\phi_F)$ be a generic representation of $\GL_n(F)$, then $\phi_F$ is symplectic if and only $\pi$ is $\Psi_F$-distinguished, whereas $\phi_F$ is orthogonal if and only if $\pi$ is $\Theta_F$-distinguished. 
\item Let $\tau=\LLC(\phi_E)$ be a generic representation of $\GL_n(E)$, then $\phi_E$ is conjugate-symplectic if and only $\tau$ is $\eta_{E/F}$-distinguished, whereas 
$\phi_E$ is conjugate-orthogonal if and only if $\tau$ is $\1_{E/F}$-distinguished. 
\end{itemize}
\end{thm}

\subsection{A reminder on epsilon factors}\label{sec eps}

Let $K'/K$ be a finite separable extension of non Archimedean local fields. We denote by $\w_K$ a uniformizer of $K$ and by $P_K$ the maximal ideal of the ring of integers $O_K$ of $K$. If $\psi$ is a non trivial character of $K$, we denote by $\psi_{K'}$ the character $\psi\circ \tr_{K'/K}$. We call \textit{the conductor of $\psi$} and write $d(\psi)$ for the smallest integer $d$ such that $\psi$ is trivial on $P_K^d$. When 
\textit{$K'/K$ is unramified}, it follows from \cite[Chapter 8, Corollary 3]{W.74} that 
\begin{equation}\label{equation equal conductors}
d(\psi_{K'})=d(\psi).\end{equation}
 Similarly if $\chi$ is a character of $W'_K$ identified by local class field theory with a character of $K^*$, we call 
the Artin conductor of $\chi$ the integer $a(\chi)$ equal to zero if $\chi$ is unramified, or equal to the smallest 
integer $a$ such that $\chi$ is trivial on $1+P_{K}^a$ if $\chi$ is ramified. More generally one can define 
the Artin conductor $a(\phi)$ (which is an integer) of any representation $\phi$ of $W'_K$, see \cite[3.4.5]{T.79} when $\phi$ is a representation of $W_K$ and 
\cite[Section 2.2, (10)]{GR} in general. The Artin conductor is additive: \[a(\phi\oplus\phi')=a(\phi)+a(\phi')\] for 
$\phi$ and $\phi'$ in $\Rep(W'_K)$. If $\phi$ is a representation of $W'_K$, and $\psi$ is a non trivial character of $K$, we refer to \cite[3.6.4]{T.79} and \cite[31.3]{BH.06} or \cite[Section 2.2]{GR} for the definition of the root number 
$\e(1/2,\phi,\psi)$. 
One then defines the Langlands $\l$-constant: \[\l({K'}/K,\psi)=\frac{\e(1/2,\Ind_{W_{K'}}^{W_K}(\mathbf{1}_{W_K}),\psi)}{\e(1/2,\mathbf{1}_{W_K'},\psi_{K'})}.\] For $a\in K^\times$, we set $\psi_a=\psi(a\ . \ )$. These constants enjoy the following list of properties, which we will freely use later in the paper.

\begin{enumerate}
\item \label{equation additivity of epsilon} $\e(1/2,\phi\oplus \phi',\psi)=\e(1/2,\phi,\psi)\e(1/2,\phi',\psi)$ where $\phi'$ is another representation of $W'_K$ (\cite[(3.4.2)]{T.79}).
\item \label{equation translation caractère additif} $\e(1/2,\phi,\psi_a)=\det(\phi(a))\e(1/2,\phi,\psi)$ (\cite[(3.6.6)]{T.79}).
\item \label{equation epsilon times epsilon dual} $\e(1/2,\phi,\psi)^2=\det(\phi)(-1)$ when $\phi$ is self-dual (\cite[Section 2.3, (11)]{GR}).
\item \label{equation torsion NR} If $d(\psi)=0$ and $\mu$ is an unramified character of $K^*$, it follows from \cite[Section 2.3, (9)]{GR} that: 
\[\e(1/2,\mu\otimes \phi,\psi)=\mu(\w_K^{a(\phi)})\e(1/2,\phi,\psi).\] 
\item \label{equation GGP} If ${K'}/K$ is quadratic with $K$ of characteristic not $2$, $\d\in \ker(\tr_{{K'}/K})-\{0\}$, and $\phi$ is a $K'/K$-orthogonal representation of $W'_{K'}$, then by \cite[Proposition 5.2]{GGP} (generalizing \cite[Theorem 3]{FQ.73}): \[\e(1/2,\phi,\psi_{K'})=\det(\phi)(\d).\]
\item \label{equation inductivity} If $\phi_{K'}$ is an $r$-dimensional representation of $W'_{K'}$, then 
\[\e(1/2,\Ind_{W'_{K'}}^{W'_K}(\phi_{K'}),\psi)=\l({K'}/K,\psi)^r\e(1/2,\phi_{K'},\psi_{K'})\] 
(\cite[(30.4.2)]{BH.06}). When applied to a $K'/K$ quadratic and $\phi_{K'}=\Res_{W'_{K'}}^{W'_K}(\phi)$ for $\phi$ a representation of $W'_K$, one gets 
\[\e(1/2,\phi,\psi)\e(1/2,\eta_{K'/K}\otimes \phi,\psi)=\l({K'}/K,\psi)^r\e(1/2,\Res_{W'_{K'}}^{W'_K}(\phi),\psi_{K'})\] 
\item \label{equation constante de Langlands NR} If ${K'}/K$ is unramified with $[K'/K]=n$: \[\l({K'}/K,\psi)=(-1)^{d(\psi)(n-1)}\] (for example \cite{M.86} and \ref{equation translation caractère additif}., together with Equation (\ref{equation equal conductors})). In particular 
if $d(\psi)=0$ then \[\l({K'}/K,\psi)=1.\]
\end{enumerate}

\section{Distinction, base change, and automorphic induction}

\iffalse

\begin{lemma}
Let $\phi_F$ be a representation of $W'_F$ and $\phi_E$ a representation of $W'_E$, then $\phi_F \oplus  \phi_F^\vee$ is $\pm 1$-self-dual, whereas 
$\phi_E \oplus  (\phi_E^\vee)^s$ is $(\pm 1, s)$-self-dual.
\end{lemma}
\begin{proof}
The first classical assertion follows from the fact that for $\ve\in \{\pm 1\}$, the bilinear form 
$(x+x^\vee,y+y^\vee)\mapsto \langle x, y^\vee \rangle +\ve \langle y, x^\vee \rangle$ is $(W'_F,\ve)$-bilinear. 

For the second assertion we can see $\phi_E\oplus (\phi_E^\vee)^s$ as the $W'_E$-submodule $\varphi_E:=\phi_E\oplus s^{-1}.\phi_E^\vee$ of 
\[\Res_{W'_E}^{W'_F}(\Ind_{W'_E}^{W'_F}(\phi_E\oplus \phi_E^\vee))=(\phi_E\oplus \phi_E^\vee)\oplus s^{-1}.(\phi_E\oplus \phi_E^\vee).\] Now from the the first point and the proof of Proposition \ref{prop main} point \ref{1} we have on $\Ind_{W'_E}^{W'_F}(\phi_E\oplus \phi_E^\vee)$ a $(W'_F,\ve)$-bilinear form given by 
\[B_F((x+x^\vee)+s^{-1}.(y+y^\vee),(x'+x'^\vee)+s^{-1}.(y'+y'^\vee))=\langle x, x'^\vee \rangle +\ve \langle x', x^\vee \rangle
+\langle y, y'^\vee \rangle +\ve \langle y', y^\vee \rangle.\] 
It follows the proof of Proposition \ref{prop main} point \ref{3} that \[B_E(x+s^{-1}.y^\vee,x'+s^{-1}.y'^\vee)=B_{F}(x+s^{-1}.y^\vee,s^{-1}.(x'+s^{-1}.y'^\vee))\] is a $(E/F,\ve)$-bilinear form on $\varphi_E$ (it is non degenerate because 
$B_{E}(x+s^{-1}.y^\vee,x'+s^{-1}.y'^\vee)=\langle x, s^{-2}.y'^\vee\rangle +\epsilon \langle x', y^\vee\rangle$).
\end{proof}

\fi

From now on $E/F$ is a separable quadratic extension of non Archimedean local fields. Our main result is the following proposition, and we notice that half of its first point is 
\cite[Lemma 3.5, (i)]{GGP}.

\begin{prop}\label{prop main}
\begin{enumerate}
\item\label{1} Let $\phi_E$ be a semi-simple representation of $W'_E$ which is either $\ve$-self-dual or $(E/F,\ve)$-dual, then $\Ind_{W'_E}^{W'_F}(\phi_E)$ is
$\ve$-selfudal. 
\item\label{2} Conversely if $\phi_E$ is irreducible and $\Ind_{W'_E}^{W'_F}(\phi_E)$ is $\ve$-self-dual:
\begin{enumerate}
\item if $\Ind_{W'_E}^{W'_F}(\phi_E)$ is irreducible, i.e. $\phi_E^s\not\simeq \phi_E$, then either $\phi_E$ is $\ve$-self-dual or $(E/F,\ve)$-dual, but not both together, 
\item if $\Ind_{W'_E}^{W'_F}(\phi_E)$ is reducible, i.e. $\phi_E^s\simeq \phi_E$, then 
$\phi_E$ is both $\ve$-self-dual and $(E/F,\ve)$-dual.
\end{enumerate}
\item\label{3} Let $\phi_F$ be a semi-simple representation of $W'_F$ which is $\ve$-self-dual, then 
$\Res_{W'_E}^{W'_F}(\phi_F)$ is $\ve$-self-dual and $(E/F,\ve)$-dual. 
\item\label{4} Conversely, if $\phi_F$ is irreducible and $\Res_{W'_E}^{W'_F}(\phi_F)$ is $\ve$-self-dual and $(E/F,\ve)$-dual then 
 $\phi_F$ is $\ve$-self-dual.
\end{enumerate}
\end{prop}
\begin{proof}
\begin{enumerate}
 \item  First suppose that $B_E$ is a $(E/F,\ve)$-bilinear form on $\phi_E$. Write an element $v$ (resp. $v'$) in $\Ind_{W'_E}^{W'_F}(\phi_E)$ under the form 
 $v=x+s^{-1}.y$ (resp. $v'=x'+s^{-1}.y'$) for $x,\ x', \ y, \ y'$ in $\phi_E$, and set \[B_F(v,v')=B_E(x,y')+ \ve B_E(x',y).\] Then $B_F$ is $W'_E$-invariant because 
 $B_E$ is $(W'_E,\ve)$-conjugate (it is non-degenerate because so is $B_E$). Finally 
 \[B_F(s .v,s .v')=B_E(y,s^{2}.x')+\ve B_E(y',s^{2}.x)=\ve B_E(x',y)+ B_E(x,y')=B_F(v,v').\] 
 Similarly if $B_E$ is $(W'_E,\ve)$-bilinear, then one checks that 
 \[B_F(x+s^{-1}y,x'+s^{-1}.y')=B_E(x,x')+ B_E(y,y')\] defines a $(W'_F,\ve)$-bilinear form on $\phi_F$.
 \item Suppose that $\phi_E$ is irreducible and that $\Ind_{W'_E}^{W'_F}(\phi_E)$ is $\ve$-self-dual with $(W'_F,\ve)$-bilinear form $B_F$.
 \begin{enumerate}
 \item If $\phi_E^s\not \simeq \phi_E$, because $\Ind_{W'_E}^{W'_F}(\phi_E)$ is self-dual then either $\phi_E$ is self-dual, or 
 $\phi_E^s\simeq \phi_E^\vee$ but not both together. In the first case, say that $\phi_E$ is $\ve'$-self-dual, then so is 
 $\Ind_{W'_E}^{W'_F}(\phi_E)$ by \ref{1}, but then $\ve'=\ve$ by irreducibility of $\Ind_{W'_E}^{W'_F}(\phi_E)$. If 
 $\phi_E^s\simeq \phi_E^\vee$ we conclude in a similar manner.
 \item If $\phi_E^s\simeq\phi_E$ then $\Ind_{W'_E}^{W'_F}(\phi_E)\simeq \phi\oplus \eta_{E/F}\otimes \phi$ for $\phi$ extending $\phi_E$, and 
 $\phi\not \simeq  \eta_{E/F}\otimes\phi$. Because $\phi\not \simeq  \eta_{E/F}\otimes\phi$ there are two disjoint cases. The first is when $\phi$ is self-dual, in which case $\phi\perp \eta_{E/F}\otimes\phi$ and $B_F$ restricts non trivially to $\phi\times \phi$ (and $\eta_{E/F}\otimes\phi\times \eta_{E/F}\otimes\phi$). Then $\phi_E$ is $\ve$-dual 
 and $(\ve,s)$-dual by \ref{3}. Otherwise $\phi^\vee\simeq \eta_{E/F}\otimes\phi$ and $B_F$ is zero on $\phi\times \phi$ and $\eta_{E/F}\otimes\phi\times \eta_{E/F}\otimes\phi$. In this case there is up to scaling a unique $W'_F$-invariant bilinear form on $\Ind_{W'_E}^{W'_F}(\phi_E)$, namely $B_F$. Because $\phi_E^\vee\simeq \phi_E$ (by restricting the relation $\phi^\vee\simeq \eta_{E/F}\otimes\phi$ to $W'_E$), $\phi_E$ must be 
 $\ve'$-self-dual, hence $\Ind_{W'_E}^{W'_F}(\phi_E)$ as well by \ref{1}, but then we have $\ve'=\ve$ by multiplicity one of 
 $W'_F$-invariant bilinear form on $\Ind_{W'_E}^{W'_F}(\phi_E)$. Moreover because 
 $\phi_E^s=\phi_E$ the parameter $\phi_E$ is also $(\ve'',s)$-self-dual and by \ref{1} again we deduce that $\ve''=\ve$. 
 \end{enumerate}
 \item Let $B_F$ be a $(W'_F,\ve)$-bilinear form on $\phi_F$, then it remains a $(W'_E,\ve)$-bilinear on $\Res_{W'_E}^{W'_F}(\phi_F)$, and on the other hand 
 \[B_E(x,y)=B_F(x,s^{-1}.y)\] is an $(E/F,\ve)$-bilinear form on $\Res_{W'_E}^{W'_F}(\phi_F)$.
 \item We suppose that $\phi_F$ is irreducible and that $\Res_{W'_E}^{W'_F}(\phi_F)$ is $\ve$-self-dual and also $(E/F,\ve)$-dual. There are two cases to consider. 
 
 First if $\Res_{W'_E}^{W'_F}(\phi_F)$ is irreducible, then denote by $B_E$ the $(W'_E,\ve)$-bilinear form on 
 $\Res_{W'_E}^{W'_F}(\phi_F)$. Now set $D_E(x,y)=B_E(x,s^{-1}.y)$ for $x,\ y \in \Res_{W'_E}^{W'_F}(\phi_F)$. Clearly $D_E$ is $E/F$-bilinear, but by irreducibility $\Res_{W'_E}^{W'_F}(\phi_F)$ affords at most one such form up to scalar, hence $D_E$ must be $(E/F,\ve)$-bilinear. This 
 implies that for $x$ and $y$ in $\Res_{W'_E}^{W'_F}(\phi_F)$ one has
  \[B_E(s.x,s.y)=D_E(s.x,s^{2}.y)=\ve D_E(y,s.x)=\ve B_E(y,x)=B_E(x,y).\]
  All in all, when $\Res_{W'_E}^{W'_F}(\phi_F)$ is irreducible we deduce that $B_E$ is in fact $W_F$-invariant hence that 
  $\phi_F$ is $\ve$-self-dual. 
  
It remains to treat the case where $\Res_{W'_E}^{W'_F}(\phi_F)$ is reducible. In this case it is of the form $\phi_E\oplus s^{-1}.\phi_E$ where $\phi_E$ is an irreducible of $W'_E$ such that $\phi_E^s\not \simeq \phi_E$ and $\phi_F=\Ind_{W'_E}^{W'_F}(\phi_E)$. First because $\Res_{W'_E}^{W'_F}(\phi_F)$ is $\ve$-self-dual, then 
the $(W'_E,\ve)$-bilinear form $B_E$ on $\Res_{W'_E}^{W'_F}(\phi_F)$ either induces an isomorphism $\phi_E^s\simeq \phi_E^\vee$ or $\phi_E \perp s^{-1}.\phi_E$ for $B_E$. 
Similarly the $(E/F,\ve)$-bilinear form $C_E$ on $\Res_{W'_E}^{W'_F}(\phi_F)$ either induces an isomorphism $(\phi_E^s)^\vee\simeq \phi_E^s \Longleftrightarrow \phi_E\simeq \phi_E^\vee$ or $\phi_E \perp s^{-1}.\phi_E$ for $C_E$. Suppose that $B_E$ induces 
an isomorphism $\phi_E^s\simeq \phi_E^\vee$, then one must have $\phi_E \perp s^{-1}.\phi_E$ for $C_E$ because $\phi_E\not\simeq \phi_E^s\simeq \phi_E^\vee$. This implies that $C_E$ induces an $(E/F,\ve)$-bilinear form on $\phi_E$ and by point \ref{1} we deduce that $\phi_F$ is $\e$-self-dual. 
On the other hand if $\phi_E \perp s^{-1}.\phi_E$ for $B_E$ then $B_E$ 
induces an $(W'_E,\ve)$-bilinear form on $\phi_E$ and $\phi_F$ is $\ve$-self-dual again by point \ref{1}.
\end{enumerate}

\end{proof}

\textit{Supposing that $F$ has characteristic zero}, we translate Proposition \ref{prop main} via the LLC, in view of the results recalled in Section \ref{section LLC}. For this we denote by $\sigma$ the Galois conjugation of $E/F$ and its extension to $\GL_n(E)$, and set $\tau^\sigma=\tau\circ \sigma$ for any representation of $\GL_n(E)$. 

\begin{cor}\label{cor main}
\begin{enumerate}
\item Let $\tau$ be an irreducible representation of $\GL_n(E)$ such that $\AI_E^F(\tau)$ is generic (for example $\tau$ generic unitary). If $\tau$ is either $\Theta_E$-distinguished or $\1_{E/F}$-distinguished, then $\AI_E^F(\tau)$ is $\Theta_F$-distinguished, whereas if $\tau$ is either $\Psi_E$-distinguished or $\eta_{E/F}$-distinguished, then $\AI_E^F(\tau)$ is $\Psi_F$-distinguished.
\item Conversely if $\tau$ is a discrete series representation of $\GL_n(E)$.
\begin{enumerate}
\item Suppose that $\AI_E^F(\tau)$ is $\Psi_F$-distinguished:
\begin{enumerate}
\item if $\AI_E^F(\tau)$ is a discrete series, i.e. if $\tau^\sigma\not \simeq \tau$, then either $\tau$ is $\Psi_E$-distinguished or $\eta_{E/F}$-distinguished, but not both together, 
\item if $\AI_E^F(\tau)$ is not a discrete series, i.e. $\tau^\sigma\simeq  \tau$, then 
$\tau$ is both $\Psi_E$-distinguished and $\eta_{E/F}$-distinguished.
\end{enumerate}
\item Suppose that $\AI_E^F(\tau)$ is $\Theta_F$-distinguished:
\begin{enumerate}
\item if $\AI_E^F(\tau)$ is a discrete series, i.e. $\tau^\sigma\not \simeq \tau$, then either $\tau$ is $\Theta_F$-distinguished or $\1_{E/F}$-distinguished, but not both together, 
\item if $\AI_E^F(\tau)$ is not a discrete series, i.e. $\tau^\sigma\simeq \tau$, then 
$\tau$ is both $\Theta_E$-distinguished and $\1_{E/F}$-distinguished.
\end{enumerate}
\end{enumerate}
\item Let $\pi$ be an irreducible representation of $\GL_n(F)$ such that $\BC_F^E(\pi)$ is generic (for example $\pi$ generic unitary). If $\pi$ is $\Theta_F$-distinguished, then 
$\BC_F^E(\pi)$ is $\Theta_E$-distinguished and $\1_{E/F}$-distinguished, whereas if $\pi$ is $\Psi_F$-distinguished, then 
$\BC_F^E(\pi)$ is $\Psi_E$-distinguished and $\eta_{E/F}$-distinguished.
\item Conversely suppose that $\pi$ is a discrete series. If $\BC_F^E(\pi)$ is $\Theta_E$-distinguished and $\1_{E/F}$-distinguished, then 
 $\pi$ is $\Theta_F$-distinguished, whereas if $\BC_F^E(\pi)$ is $\Psi_E$-distinguished and $\eta_{E/F}$-distinguished, then 
 $\pi$ is $\Psi_F$-distinguished.
\end{enumerate}
\end{cor}

\section{Parity of the Artin conductor of self-dual representations}

In this section $F$ is again a non Archimedean local field. First, using \cite[Proposition 5.2]{GGP} (which is itself a 
quick but non trivial consequence of a difficult result of Deligne \cite{Deligne} on root numbers of orthogonal representations), we quickly 
recover in odd residual characteristic from Proposition \ref{prop main}, \ref{3} the following result due to Serre \cite{Serre} (the result in question also holds in even residual characteristic by \cite{Serre}). In other words we show that the result of \cite{Deligne} implies that of \cite{Serre} for non Archimedean local fields of odd residual characteristic.

\begin{cor}[of Proposition \ref{prop main}, \cite{Serre}]\label{cor serre}
Let $\phi$ be an orthogonal representation of $W'_F$. We have the following congruence of Artin conductors: $a(\phi)=a(\det(\phi))[2]$.
\end{cor}
\begin{proof}
As we said the result is true for $F$ of any residual characteristic, and we recover it in this proof for $F$ of residual characteristic different from $2$. Let $E$ be the unramified quadratic extension of $F$, and take $\psi$ a character of $F$ of conductor zero. We have according to Section \ref{sec eps}, Points \ref{equation inductivity} and \ref{equation constante de Langlands NR} 
\begin{equation}\label{eq chakal} \e(1/2,\Res_{W'_E}^{W'_F}(\phi),\psi_E)=\e(1/2,\phi,\psi)\e(1/2,\eta_{E/F}\otimes \phi,\psi).\end{equation} Now denoting by $q$ the residual cardinality of $F$, let $u$ be an element of order $q^2-1$ in $E^\times$, so that 
$\d:=u^{(q+1)/2}$ does not belong to $F$ but $\Delta:=\d^2$ belongs to $F$. Note that the image of $\Delta$ generates $O_F^\times/1+P_F$. Then $\e(1/2,\Res_{W'_E}^{W'_F}(\phi),\psi_{E,\delta}^{-1})=1$ by Proposition \ref{prop main}, \ref{3} and Section \ref{sec eps}, Point \ref{equation GGP}, hence 
\[\e(1/2,\Res_{W'_E}^{W'_F}(\phi),\psi_E)=\det(\Res_{W'_E}^{W'_F}(\phi))(\delta)=\det(\phi)(N_{E/F}(\delta))=\det(\phi)(-\Delta)\] thanks to 
Section \ref{sec eps}, Point \ref{equation translation caractère additif}. Now observe 
that $\det(\phi)$ is quadratic as $\phi$ is self-dual, but 
because $q$ is odd it is trivial on $1+P_F$, hence it has conductor $0$ or $1$, and it is of conductor zero if and only if 
$\det(\phi)(\Delta)=1$, hence $\det(\phi)(\Delta)=(-1)^{a(\det(\phi))}$, so 
\[\e(1/2,\Res_{W'_E}^{W'_F}(\phi),\psi_E)=(-1)^{a(\det(\phi))}\det(\phi)(-1).\] Now 
$\e(1/2,\eta_{E/F}\otimes \phi,\psi)=(-1)^{a(\phi)}\e(1/2,\phi,\psi)$ thanks to Section \ref{sec eps}, Point \ref{equation torsion NR}, hence Section \ref{sec eps}, Point \ref{equation epsilon times epsilon dual} implies the following:
\[\e(1/2,\phi,\psi)\e(1/2,\eta_{E/F}\otimes \phi,\psi)=(-1)^{a(\phi)}\e(1/2,\phi,\psi)^2
=(-1)^{a(\phi)}\det(\phi)(-1).\] The result now follows from Equation (\ref{eq chakal}).
\end{proof}

One can legitimately ask about the parity of the Artin conductor of symplectic representations of $W'_F$. The answer seems much more complicated, and one way to adress it is via the LLC, using the so called Prasad and Takloo-Bighash conjecture, which is now a theorem when $F$ has characteristic zero and residual characteristic different from $2$ (\cite{X.20}, \cite{Sec.20}, \cite{Suz.20}, \cite{SX.20}). To this end we recall that for $E/F$ a separable quadratic extension, then the matrix algebra $\M_{n}(E)$ embeds uniquely up to $\GL_{2n}(F)$-conjugacy into $\M_{2n}(F)$ as an $F$-subalgebra by the Skolem-Noether theorem. We fix such an embedding, which in turn gives rise to an embedding of $\GL_n(E)$ into $\GL_{2n}(F)$. We then say that an irreducible representation $\pi$ of $\GL_{2n}(F)$ is 
$\1^{E/F}$-distinguished if and only if $\Hom_{\GL_n(E)}(\pi,\1)\neq\{0\}$. We recall the following theorem, which is a consequence of one part of the Prasad and Takloo-Bighash conjecture.

\begin{thm}[\cite{X.20}, \cite{Sec.20}, \cite{SX.20}]\label{thm PTB}
Suppose that $F$ has characteristic zero and residual characteristic different from $2$. If $\phi$ is an irreducible symplectic representation of $W'_F$ of dimension $2n$, then \[\e(1/2,\phi\otimes\Ind_{W'_E}^{W'_F}(\1))=\eta_{E/F}(-1)^{n}\] if $\LLC(\phi)$ is $\1^{E/F}$-distinguished and \[\e(1/2,\phi\otimes\Ind_{W'_E}^{W'_F}(\1))=-\eta_{E/F}(-1)^{n}\] otherwise. 
\end{thm}
\begin{rem}
In the statement above, as the determinant of a symplectic representation is equal to $1$, we suppressed the dependence of the root number 
$\e(1/2,\phi\otimes\Ind_{W'_E}^{W'_F}(\1),\psi)$ on the non-trivial additive character $\psi$ of $F$. 
\end{rem}

As an immediate corollary we obtain the following result on the parity of Artin conductors of symplectic representations. 

\begin{cor}\label{cor PTB}
Suppose that $F$ has characteristic zero and residual characteristic different from $2$, denote by $E$ the unramified quadratic extension of $F$, and let $\phi$ be an irreducible symplectic representation of $W'_F$. Then $a(\phi)$ is even if and only if $\LLC(\phi)$ is $\1^{E/F}$-distinguished.
\end{cor}
\begin{proof}
It easily follows, along the lines of the proof of Corollary \ref{cor serre}, from Theorem \ref{thm PTB}, noting that $\eta_{E/F}(-1)=1$.
\end{proof}

\begin{rem}
A general symplectic representation $\phi$ of $W'_F$ being a direct sum of the form 
$\oplus_{i=1}^r\phi_i \oplus_{j=1}^s(\phi'_j\oplus {\phi'_j}^\vee)$ for $\phi_i$ irreducible symplectic and $\phi'_j$ irreducible, we deduce the parity of $a(\phi)$ from Corollary \ref{cor PTB} and such a decomposition. Namely, by Corollary \ref{cor serre} $a(\phi'_j\oplus {\phi'}_j^\vee)\equiv 0[2]$. Hence setting $\epsilon_i\in \{\pm 1\}$ being equal to $1$ if and only if $\LLC(\phi_i)$ is $\1^{E/F}$-distinguished, we deduce by additivity of the Artin conductor that 
$(-1)^{a(\phi)}=\prod_{i=1}^r \epsilon_i$. 
\end{rem}

\begin{rem}
Looking at it from another angle, one sees that a symplectic discrete series representation of $\GL_{2n}(F)$ is $\1^{E/F}$-distinguished ($E/F$ unramified) if and only if it has even conductor.
\end{rem}

\bibliographystyle{alpha}
\bibliography{ptb}

\end{document}